\documentclass[12pt]{article}

\usepackage{amsmath, epsfig, cite}
\usepackage{amssymb}
\usepackage{amsfonts}
\usepackage{latexsym}
\usepackage{amsthm}

\newtheorem{thm}{Theorem}[section]
\newtheorem{prop}[thm]{Proposition}

\newtheorem{cor}[thm]{Corollary}

\numberwithin{equation}{section}

\renewcommand{\thefootnote}{}

\begin{document}

\begin{center}
{\large\bf A further $q$-analogue of Van Hamme's (H.2)
supercongruence for $p\equiv1\pmod{4}$
 \footnote{This work is supported by the National Natural Science
Foundation of China (No. 11661032).}}
\end{center}

\renewcommand{\thefootnote}{$\dagger$}

\vskip 2mm \centerline{Chuanan Wei}
\begin{center}
{School of Biomedical Information and Engineering,\\ Hainan Medical University, Haikou 571199, China\\
{\tt weichuanan78@163.com  } }
\end{center}


\vskip 0.7cm \noindent{\bf Abstract.} Several years ago, Long and
Ramakrishna [Adv. Math. 290 (2016), 773--808] extended Van Hamme's
(H.2) supercongruence to the modulus $p^3$ case. Recently, Guo [Int.
J. Number Theory, to appear] found a $q$-analogue of the
Long--Ramakrishna formula for $p\equiv 3\pmod 4$. In this note, a
$q$-analogue of the Long--Ramakrishna formula for $p\equiv 1\pmod 4$
is derived through the $q$-Whipple formulas and the Chinese
remainder theorem for coprime polynomials.

\vskip 3mm \noindent {\it Keywords}: basic hypergeometric series;
$q$-Whipple formula; $q$-supercongruence

 \vskip 0.2cm \noindent{\it AMS
Subject Classifications:} 33D15; 11A07; 11B65

\section{Introduction}
For any complex variable $x$, define the shifted-factorial to be
\[(x)_{0}=1\quad \text{and}\quad (x)_{n}
=x(x+1)\cdots(x+n-1)\quad \text{when}\quad n\in\mathbb{N}.\]
 In 1997, Van Hamme
\cite[(H.2)]{Hamme} conjectured that
\begin{equation}\label{eq:hamme}
\sum_{k=0}^{(p-1)/2}\frac{(1/2)_k^3}{k!^3}\equiv
\begin{cases} \displaystyle -\Gamma_p(1/4)^4  \pmod{p^2}, &\text{if $p\equiv 1\pmod 4$,}\\[10pt]
 0\pmod{p^2}, &\text{if $p\equiv 3\pmod 4$.}
\end{cases}
\end{equation}
Here and throughout the paper, $p$ always denotes an odd prime and
$\Gamma_p(x)$ is the $p$-adic Gamma function. In 2016, Long and
Ramakrishna \cite[Theorem 3]{LR} gave the following extension of
\eqref{eq:hamme}:
\begin{equation}\label{eq:long}
\sum_{k=0}^{(p-1)/2}\frac{(1/2)_k^3}{k!^3}\equiv
\begin{cases} \displaystyle -\Gamma_p(1/4)^4  \pmod{p^3}, &\text{if $p\equiv 1\pmod 4$,}\\[10pt]
 \displaystyle -\frac{p^2}{16}\Gamma_p(1/4)^4\pmod{p^3}, &\text{if $p\equiv 3\pmod 4$.}
\end{cases}
\end{equation}

 For any complex numbers $x$ and $q$, define the $q$-shifted factorial
 as
 \begin{equation*}
(x;q)_{0}=1\quad\text{and}\quad
(x;q)_n=(1-x)(1-xq)\cdots(1-xq^{n-1})\quad \text{when}\quad
n\in\mathbb{N}.
 \end{equation*}
For simplicity, we also adopt the compact notation
\begin{equation*}
(x_1,x_2,\dots,x_m;q)_{n}=(x_1;q)_{n}(x_2;q)_{n}\cdots(x_m;q)_{n}.
 \end{equation*}
Following Gasper and Rahman \cite{Gasper}, define the basic
hypergeometric series $_{r+1}\phi_{r}$ by
$$
_{r+1}\phi_{r}\left[\begin{array}{c}
a_1,a_2,\ldots,a_{r+1}\\
b_1,b_2,\ldots,b_{r}
\end{array};q,\, z
\right] =\sum_{k=0}^{\infty}\frac{(a_1,a_2,\ldots, a_{r+1};q)_k}
{(q,b_1,b_2,\ldots,b_{r};q)_k}z^k.
$$
Then the $q$-Whipple formula due to Andrews \cite{Andrews} and the
$q$-Whipple formula due to Jain \cite{Jain} can be stated as
\begin{align}
& _{4}\phi_{3}\!\left[\begin{array}{cccccccc}
 q^{-n},  q^{1+n}, b,  -b \\
 -q, c,  b^2q/c
\end{array};q,\, q \right]
=q^{\binom{n+1}{2}}\frac{(b^2q^{1-n}/c, cq^{-n};q^2)_{n}}
 {(b^2q/c, c;q)_{n}}, \label{eq:q-whipple-a}
 \\[5pt]
& _{4}\phi_{3}\!\left[\begin{array}{cccccccc}
 a,  q/a, q^{-n},  -q^{-n} \\
 c, q^{1-2n}/c,  -q
 \end{array};q,\, q \right]
=\frac{(ac, cq/a;q^2)_{n}}
 {(c;q)_{2n}}. \label{eq:q-whipple-b}
\end{align}

Recently, Guo and Zudilin \cite[Theorem 2]{GuoZu2} displayed a
$q$-analogue of \eqref{eq:hamme}: for any positive odd integer $n$,
\begin{align}
&\sum_{k=0}^{(n-1)/2}\frac{(q;q^2)_k^2(q^2;q^4)_k}{(q^2;q^2)_k^2(q^4;q^4)_k}q^{2k}
\notag\\[5pt]
&\equiv
\begin{cases} \displaystyle\frac{(q^2;q^4)_{(n-1)/4}^2}{(q^4;q^4)_{(n-1)/4}^2}\pmod{\Phi_n(q)^2}, &\text{if $n\equiv 1\pmod 4$,}\\[10pt]
 \displaystyle 0\pmod{\Phi_n(q)^2}, &\text{if $n\equiv 3\pmod 4$.}
\end{cases}
\label{eq:guo-a}
\end{align}
 Here and throughout the
paper, $\Phi_n(q)$ stands for  the $n$-th cyclotomic polynomial in
$q$:
\begin{equation*}
\Phi_n(q)=\prod_{\substack{1\leqslant k\leqslant n\\
\gcd(k,n)=1}}(q-\zeta^k),
\end{equation*}
where $\zeta$ is an $n$-th primitive root of unity. Further, Guo
\cite[Theorem 1]{Guo-new} provided the following partial
$q$-analogue of \eqref{eq:long}: for any positive integer
$n\equiv3\pmod{4}$,
\begin{align}
\sum_{k=0}^{(n-1)/2}\frac{(q,;q^2)_k^2(q^2;q^4)_k}{(q^2;q^2)_k^2(q^4;q^4)_k}q^{2k}
\equiv[n]\frac{(q^3;q^4)_{(n-1)/2}}{(q^5;q^4)_{(n-1)/2}}\pmod{\Phi_n(q)^3}.
\label{eq:guo-b}
\end{align}
For more $q$-analogues of supercongruences, we refer the reader to
\cite{Guo-rima,Guo-jmaa,Guo-rama,Guo-a2,GS1,GuoZu,LP,NP,Tauraso,WY-a,Zu19}.

Motivated by the work just mentioned, we
shall establish the following result.

\begin{thm}\label{thm-a}
Let $n\equiv 1\pmod 4$ be a positive integer. Then, modulo
$\Phi_n(q)^3$,
\begin{align*}
\sum_{k=0}^{(n-1)/2}\frac{(q;q^2)_k^2(q^2;q^4)_k}{(q^2;q^2)_k^2(q^4;q^4)_k}q^{2k}
\equiv
q^{(n-1)/2}\frac{(q^2;q^4)_{(n-1)/4}^2}{(q^4;q^4)_{(n-1)/4}^2}\bigg\{1+2[n]^2\sum_{i=1}^{(n-1)/4}\frac{q^{4i-2}}{[4i-2]^2}\bigg\}.
\end{align*}
\end{thm}

Obviously, Theorem \ref{thm-a} is an extension of \eqref{eq:guo-a}
for $n\equiv1\pmod{4}$. Letting $n=p$ be a prime and taking $q\to 1$
in this theorem, we obtain the conclusion:
\begin{align}\label{eq:wei}
\sum_{k=0}^{(p-1)/2}\frac{(1/2)_k^3}{k!^3} \equiv
\frac{(1/2)_{(p-1)/4}^2}{\big((p-1)/4\big)!^2}\bigg\{1+\frac{p^{2}}{2}H_{(p-1)/2}^{(2)}-\frac{p^{2}}{8}H_{(p-1)/4}^{(2)}\bigg\}\pmod{p^3},
\end{align}
where the harmonic numbers of $2$-order are given by
\[H_{m}^{(2)}
  =\sum_{k=1}^m\frac{1}{k^{2}}.\]
Using the known formula (cf. \cite[Page 7]{Sun}):
\begin{align*}
&H_{(p-1)/2}^{(2)}\equiv0\pmod{p}\quad\text{with}\quad p>3,
\end{align*}
we deduce the following supercongruence from \eqref{eq:wei}.

\begin{cor}\label{cor-a}
Let $p\equiv 1\pmod 4$ be a prime. Then
\begin{align}\label{eq:wei-a}
\sum_{k=0}^{(p-1)/2}\frac{(1/2)_k^3}{k!^3} \equiv
\frac{(1/2)_{(p-1)/4}^2}{\big((p-1)/4\big)!^2}\bigg\{1-\frac{p^{2}}{8}H_{(p-1)/4}^{(2)}\bigg\}\pmod{p^3}.
\end{align}
\end{cor}

For the sake of explaining the equivalence of \eqref{eq:long} for
$p\equiv 1\pmod 4$ and \eqref{eq:wei-a}, we need to verify the
following relation.

\begin{prop}\label{prop-a}
Let $p\equiv 1\pmod 4$ be an odd prime. Then
\begin{align*}
\frac{(1/2)_{(p-1)/4}^2}{\big((p-1)/4\big)!^2}\bigg\{1-\frac{p^{2}}{8}H_{(p-1)/4}^{(2)}\bigg\}\equiv
-\Gamma_p(1/4)^4\pmod{p^3}.
\end{align*}
\end{prop}

The rest of the paper is arranged as follows. By means of the
Chinese remainder theorem for coprime polynomials, a
$q$-supercongruence modulo $(1-aq^n)(a-q^n)(b-q^n)$ will be derived
in Section 2. Then it is utilized to provide a proof of Theorem
\ref{thm-a} in the same section. Finally, the proof of Proposition
\ref{prop-a} will be given in Section 3.

\section{Proof of Theorem \ref{thm-a}}

In order to prove Theorem \ref{thm-a}, we need the following
parameter extension of it.

\begin{thm}\label{thm-b}
Let $n\equiv 1\pmod 4$ be a positive integer. Then, modulo
$(1-aq^n)(a-q^n)(b-q^n)$,
\begin{align}
\sum_{k=0}^{(n-1)/2}\frac{(aq,q/a,q/b,-q/b;q^2)_k}{(q^2,q^2,-q^2,q^2/b^2;q^2)_k}q^{2k}
\equiv \Omega_n(a,b),\label{eq:wei-aa}
\end{align}
where
\begin{align*}
\Omega_n(a,b)&=\frac{(b-q^n)(ab-1-a^2+aq^n)}{(a-b)(1-ab)}\frac{(b/q)^{(1-n)/2}(q^2,b^2q^2;q^4)_{(n-1)/4}}{(q^4,q^4/b^2;q^4)_{(n-1)/4}}
\\[5pt]
&+\frac{(1-aq^n)(a-q^n)}{(a-b)(1-ab)}\frac{(aq^3,q^3/a;q^4)_{(n-1)/2}}{(q^2;q^2)_{n-1}}.
\end{align*}
\end{thm}

\begin{proof}
When $a=q^{-n}$ or $a=q^n$, the left-hand side of  \eqref{eq:wei-aa}
is equal to
\begin{align}
\sum_{k=0}^{(n-1)/2}\frac{(q^{1-n},q^{1+n},q/b,-q/b;q^2)_k}{(q^2,q^2,-q^2,q^2/b^2;q^2)_k}q^{2k}
= {_{4}\phi_{3}}\!\left[\begin{array}{cccccccc}
 q^{1-n},  q^{1+n}, q/b,  -q/b \\
 q^2, -q^2,  q^2/b^2
\end{array};q^2,\, q^2 \right].
 \label{eq:whipple-aa}
\end{align}
According to \eqref{eq:q-whipple-a}, the right-hand side of
\eqref{eq:whipple-aa} can be expressed as
\begin{align*}
(b/q)^{(1-n)/2}\frac{(q^2,b^2q^2;q^4)_{(n-1)/4}}{(q^4,q^4/b^2;q^4)_{(n-1)/4}}.
\end{align*}
Since $(1-aq^n)$ and $(a-q^n)$ are relatively prime polynomials, we
get the following result: Modulo $(1-aq^n)(a-q^n)$,
\begin{align}
\sum_{k=0}^{(n-1)/2}\frac{(aq,q/a,q/b,-q/b;q^2)_k}{(q^2,q^2,-q^2,q^2/b^2;q^2)_k}q^{2k}\equiv
 (b/q)^{(1-n)/2}\frac{(q^2,b^2q^2;q^4)_{(n-1)/4}}{(q^4,q^4/b^2;q^4)_{(n-1)/4}}. \label{eq:wei-bb}
\end{align}

When $b=q^{n}$, the left-hand side of  \eqref{eq:wei-aa} is equal to
\begin{align}
\sum_{k=0}^{(n-1)/2}\frac{(aq,q/a,q^{1-n},-q^{1-n};q^2)_k}{(q^2,q^2,-q^2,q^{2-2n};q^2)_k}q^{2k}
= {_{4}\phi_{3}}\!\left[\begin{array}{cccccccc}
 aq,  q/a, q^{1-n},  -q^{1-n} \\
 q^2, -q^2,  q^{2-2n}
\end{array};q^2,\, q^2 \right].
 \label{eq:whipple-bb}
\end{align}
In terms of \eqref{eq:q-whipple-b}, the right-hand side of
\eqref{eq:whipple-bb} can be written as
\begin{align*}
\frac{(aq^3,q^3/a;q^4)_{(n-1)/2}}{(q^2;q^2)_{n-1}}.
\end{align*}
Therefore, we are led to the following conclusion: Modulo $(b-q^n)$,
\begin{align}
\sum_{k=0}^{(n-1)/2}\frac{(aq,q/a,q/b,-q/b;q^2)_k}{(q^2,q^2,-q^2,q^2/b^2;q^2)_k}q^{2k}
 \equiv\frac{(aq^3,q^3/a;q^4)_{(n-1)/2}}{(q^2;q^2)_{n-1}}.\label{eq:wei-cc}
\end{align}

It is clear that the polynomials $(1-aq^n)(a-q^n)$ and $(b-q^n)$ are
relatively prime. Noting the $q$-congruences
\begin{align*}
&\frac{(b-q^n)(ab-1-a^2+aq^n)}{(a-b)(1-ab)}\equiv1\pmod{(1-aq^n)(a-q^n)},
\\[5pt]
&\qquad\qquad\frac{(1-aq^n)(a-q^n)}{(a-b)(1-ab)}\equiv1\pmod{(b-q^n)}
\end{align*}
and employing the Chinese remainder theorem for coprime polynomials,
we deduce Theorem \ref{thm-b} from \eqref{eq:wei-bb} and
\eqref{eq:wei-cc}.
\end{proof}

\begin{proof}[Proof of Theorem \ref{thm-a}]
It is not difficult to see that
\begin{align*}
(q^2;q^2)_{n-1}&=(q^2, q^{n+1}, q^4, q^{n+3};q^4)_{(n-1)/4}
\\[5pt]
&=q^{(n-1)(3n-1)/4}(q^2, q^4, q^{2-2n}, q^{4-2n};q^4)_{(n-1)/4}\\[5pt]
&\equiv
b^{n-1}q^{(1-n^2)/4}(q^{2},q^{4},q^2/b^2,q^4/b^2;q^4)_{(n-1)/4}\pmod{(b-q^n)},
\\[5pt]
(aq^3;q^4)_{(n-1)/2}&=(aq^3;q^4)_{(n-1)/4}(aq^{n+2};q^4)_{(n-1)/4}
\\[5pt]
&\equiv(abq^{3-n};q^4)_{(n-1)/4}(abq^{2};q^4)_{(n-1)/4}
\\[5pt]
&=(-ab)^{(n-1)/4}q^{-(n-1)^2/8}(abq^2,q^2/ab;q^4)_{(n-1)/4}\pmod{(b-q^n)},
\\[5pt]
(q^3/a;q^4)_{(n-1)/2}&=(-b/a)^{(n-1)/4}q^{-(n-1)^2/8}(bq^2/a,aq^2/b;q^4)_{(n-1)/4}\pmod{(b-q^n)}.
\end{align*}
Thus, the $q$-supercongruence \eqref{eq:wei-aa} may be rewritten as
follows: Modulo $(1-aq^n)(a-q^n)(b-q^n)$,
\begin{align*}
&\sum_{k=0}^{(n-1)/2}\frac{(aq,q/a,q/b,-q/b;q^2)_k}{(q^2,q^2,-q^2,q^2/b^2;q^2)_k}q^{2k}  \\[5pt]
&\quad\equiv
\frac{(b-q^n)(ab-1-a^2+aq^n)}{(a-b)(1-ab)}\frac{(b/q)^{(1-n)/2}(q^2,b^2q^2;q^4)_{(n-1)/4}}{(q^4,q^4/b^2;q^4)_{(n-1)/4}}
\\[5pt]
&\quad\quad+\frac{(1-aq^n)(a-q^n)}{(a-b)(1-ab)}\frac{(b/q)^{(1-n)/2}(abq^2,bq^2/a,aq^2/b,q^2/ab;q^4)_{(n-1)/4}}{(q^2,q^4,q^2/b^2,q^4/b^2;q^4)_{(n-1)/4}}.
\end{align*}
Letting $b\to 1$, we arrive at the following formula: Modulo
$\Phi_n(q)(1-aq^n)(a-q^n)$,
\begin{align}
&\sum_{k=0}^{(n-1)/2}\frac{(aq,q/a;q^2)_k(q^2;q^4)_k}{(q^2;q^2)_k^2(q^4;q^4)_k}q^{2k}
\notag\\[5pt]
&\:\:\:\equiv
q^{(n-1)/2}\frac{(q^2;q^4)_{(n-1)/4}^2}{(q^4;q^4)_{(n-1)/4}^2}+q^{(n-1)/2}\frac{(1-aq^n)(a-q^n)}{(1-a)^2}
\notag\\[5pt]
&\quad\:\:\times\bigg\{\frac{(q^2;q^4)_{(n-1)/4}^2}{(q^4;q^4)_{(n-1)/4}^2}-\frac{(aq^2,q^2/a;q^4)_{(n-1)/4}^2}{(q^2,q^4;q^4)_{(n-1)/4}^2}\bigg\}.
\label{eq:wei-dd}
\end{align}
By the L'Hospital rule, we have
\begin{align*}
&\lim_{a\to1}\frac{(1-aq^n)(a-q^n)}{(1-a)^2}\bigg\{\frac{(q^2;q^4)_{(n-1)/4}^2}{(q^4;q^4)_{(n-1)/4}^2}
-\frac{(aq^2,q^2/a;q^4)_{(n-1)/4}^2}{(q^2,q^4;q^4)_{(n-1)/4}^2}\bigg\}\\[5pt]
&=2[n]^2\frac{(q^2;q^4)_{(n-1)/4}^2}{(q^4;q^4)_{(n-1)/4}^2}\sum_{i=1}^{(n-1)/4}\frac{q^{4i-2}}{[4i-2]^2}.
\end{align*}
Letting $a\to1$ in \eqref{eq:wei-dd} and utilizing the above limit,
we complete the proof of  Theorem~\ref{thm-a}.

\end{proof}

\section{Proof of Proposition \ref{prop-a}}

Via the congruence due to Wang and Pan \cite[Page 6]{Wang}:
\begin{align*}
H_{(p-1)/4}^{(2)}\equiv\frac{\Gamma_p^{''}(1/4)}{\Gamma_p(1/4)}-\bigg\{\frac{\Gamma_p^{'}(1/4)}{\Gamma_p(1/4)}\bigg\}^2\pmod{p},
\end{align*}
where $\Gamma_p^{'}(x)$ and $\Gamma_p^{''}(x)$ are respectively the
first derivative and second derivative of $\Gamma_p(x)$, we obtain
\begin{align}\label{eq:wei-b}
1-\frac{p^{2}}{8}H_{(p-1)/4}^{(2)}\equiv1-\frac{p^{2}}{8}\frac{\Gamma_p^{''}(1/4)}{\Gamma_p(1/4)}+
\frac{p^{2}}{8}\bigg\{\frac{\Gamma_p^{'}(1/4)}{\Gamma_p(1/4)}\bigg\}^2\pmod{p^3}.
\end{align}
In terms of the properties of the $p$-adic Gamma function, we get
\begin{align}
\frac{(1/2)_{(p-1)/4}^2}{\big((p-1)/4\big)!^2}&=\bigg\{\frac{\Gamma_p((1+p)/4)\Gamma_p(1)}{\Gamma_p(1/2)\Gamma_p((3+p)/4)}\bigg\}^2
\notag\\[5pt]
&=\bigg\{\frac{\Gamma_p((1+p)/4)\Gamma_p((1-p)/4)}{\Gamma_p(1/2)}\bigg\}^2
\notag\\[5pt]
&\equiv-\bigg\{\Gamma_p(1/4)+\Gamma_p^{'}(1/4)\frac{p}{4}+\Gamma_p^{''}(1/4)\frac{p^2}{2\times4^2}\bigg\}^2
\notag\\[5pt]
&\quad\times\bigg\{\Gamma_p(1/4)-\Gamma_p^{'}(1/4)\frac{p}{4}+\Gamma_p^{''}(1/4)\frac{p^2}{2\times4^2}\bigg\}^2\pmod{p^3}.
\label{eq:wei-c}
\end{align}
The combination of  \eqref{eq:wei-b} and  \eqref{eq:wei-c} produces
\begin{align*}
&\frac{(1/2)_{(p-1)/4}^2}{\big((p-1)/4\big)!^2}\bigg\{1-\frac{p^{2}}{8}H_{(p-1)/4}^{(2)}\bigg\}
 \\[5pt]
 &\quad\equiv-\bigg\{\Gamma_p(1/4)+\Gamma_p^{'}(1/4)\frac{p}{4}+\Gamma_p^{''}(1/4)\frac{p^2}{2\times4^2}\bigg\}^2
\notag\\[5pt]
&\qquad\times\bigg\{\Gamma_p(1/4)-\Gamma_p^{'}(1/4)\frac{p}{4}+\Gamma_p^{''}(1/4)\frac{p^2}{2\times4^2}\bigg\}^2
\\[5pt]
&\qquad\times\bigg\{1-\frac{p^{2}}{8}\frac{\Gamma_p^{''}(1/4)}{\Gamma_p(1/4)}+
\frac{p^{2}}{8}\bigg\{\frac{\Gamma_p^{'}(1/4)}{\Gamma_p(1/4)}\bigg\}^2\bigg\}
\\[5pt]
&\quad\equiv -\Gamma_p(1/4)^4 \pmod{p^3}.
\end{align*}


\end{document}